\documentclass[a4paper,11pt]{article}%
\usepackage{amssymb}
\usepackage{graphicx}
\usepackage{amsmath}
\usepackage{amsfonts}%
\newtheorem{theorem}{Theorem}

\newtheorem{corollary}[theorem]{Corollary}

\newtheorem{definition}{Definition}
\newtheorem{example}{Example}

\newtheorem{lemma}[theorem]{Lemma}

\newtheorem{proposition}[theorem]{Proposition}
\newtheorem{remark}{Remark}

\newenvironment{proof}[1][Proof]{\textbf{#1.} }{\ \rule{0.5em}{0.5em}}
\begin{document}

\title{A new approach to temperate generalized functions}
\author{A. Delcroix\\Equipe Analyse Alg\'{e}brique Non Lin\'{e}aire\\\textit{Laboratoire Analyse, Optimisation, Contr\^{o}le}\\Facult\'{e} des sciences - Universit\'{e} des Antilles et de la Guyane\\97159\ Pointe-\`{a}-Pitre Cedex Guadeloupe}
\maketitle

\begin{abstract}
A new approach to the algebra $\mathcal{G}_{\tau}$ of temperate nonlinear
generalized functions is proposed, in which $\mathcal{G}_{\tau}$ is based on
the space $\mathcal{O}_{M}$ endowed with is natural topology in contrary to
previous constructions.\ Thus, this construction fits perfectly in the general
scheme of construction of Colombeau type algebras and reveals better
properties of $\mathcal{G}_{\tau}$.\ This is illustrated by the natural
introduction of a regularity theory in $\mathcal{G}_{\tau}$, of the Fourier
transform, with the definition of $\mathcal{G}_{\mathcal{O}_{C}^{\prime}}$,
the space of rapidly generalized distributions which is the Fourier image of
$\mathcal{G}_{\tau}$.

\end{abstract}

\smallskip

\smallskip%

%TCIMACRO{\TeXButton{Compteurequation}{\setcounter{equation}{2}}}%
%BeginExpansion
\setcounter{equation}{2}%
%EndExpansion

\section{Introduction}

The theory of generalized functions is nowadays well established. Many
applications have been carried out in various fields of mathematics such as
partial differential equations, Lie analysis, local and microlocal analysis,
probability theory, differential geometry. (See for examples the monographies
\cite{Col1,Col2,GKOS,NePiSc,Ober1} and the references therein.)\medskip

This paper develop some remarks about a new approach to temperate generalized
functions.\ In order to justify the introduction of this new construction, we
first recall the main types of special (or simplified) algebras of generalized
based on spaces of smooth functions considered in the literature.

The original simplified \emph{Colombeau algebra of generalized functions}
$\mathcal{G}$ is based on the space $\mathcal{E}=\mathrm{C}^{\infty}$ of
smooth functions and contains the space of Schwartz distributions as a
subvector space \cite{Col1,GKOS,NePiSc,Ober1,Scarpa2}.\ The duality in the
background of this construction is, of course, $\left(  \mathcal{D}%
,\mathcal{D}^{\prime}\right)  $. As all spaces considered in the sequel,
$\mathcal{G}$ is a factor space of moderate nets modulo negligible ones, the
moderateness and the negligibility being given by the asymptotic behavior of
the nets with respect to an asymptotic scale. When an algebra containing the
space of tempered distributions is needed, the so-called \emph{algebra of
temperate generalized functions} $\mathcal{G}_{\tau}$
\cite{Col1,GKOS,Scarpa1,Scarpa2}, based on the space $\mathcal{O}_{M}$ of
slowly increasing smooth functions, is considered. The duality is in this case
$\left(  \mathcal{S},\mathcal{S}^{\prime}\right)  $. Note that this
construction is not, at first sight, related to the topology of $\mathcal{O}%
_{M}$. Finally, an algebra based on the space $\mathcal{S}$ of rapidly
decreasing functions has also been considered \cite{ADRapDec,Scarpa1,Scarpa3},
with applications (for example) in the field of pseudo differential operators
\cite{Gar,GaGrOb} or of microlocal analysis of generalized functions
\cite{ADRapDec,HorDeH,HorKun,Scarpa3}. This algebra $\mathcal{G}_{\mathcal{S}%
}$ of \emph{rapidly decreasing generalized functions} contains as a linear
subspace $\mathcal{O}_{C}^{\prime}$, the space of rapidly decreasing distributions.

The first and the last constructions are based on the natural topology of the
underlying space, which can be described by (countable) families of
semi-norms. We propose here a new version of the construction of temperate
generalized functions based on the usual topology of $\mathcal{O}_{M}$, which
therefore fits in the general scheme of construction of Colombeau type
algebras. The prize to be paid is the non countability of the family of
semi-norms defining the topology of $\mathcal{O}_{M}$.\medskip

The paper is organized as follows. Section \ref{DSHCsec2} is devoted to a
short presentation of the construction of the spaces of Colombeau type
generalized functions and of the examples quoted above. In Section
\ref{DSHCsec3}, we briefly recall the construction of the classical space of
temperate generalized functions $\mathcal{G}_{\tau}$, develop the
new\ construction and show that it leads to the same space. In Section
\ref{DSHCsec4}, we turn to the definition of the Fourier transform of elements
of $\mathcal{G}_{\tau}$.\ Using the classical theorem asserting that the
Fourier image of $\mathcal{O}_{M}$ is $\mathcal{O}_{C}^{\prime}$, we introduce
the (new) space $\mathcal{G}_{\mathcal{O}_{C}^{\prime}}$ of rapidly
generalized distributions which is the Fourier image of $\mathcal{G}_{\tau}$.
Of course, this Fourier transform will share the classical expected
properties.\ Finally, in Section \ref{DSHCsec5}, we introduce the subspace
$\mathcal{G}_{\tau}^{\infty}$ of regular elements of $\mathcal{G}_{\tau}$ and
show that the result $\mathcal{G}_{\tau}^{\infty}\cap\mathcal{S}^{\prime
}=\mathcal{O}_{M}$ holds in the spirit of the more classical one
$\mathcal{G}^{\infty}\cap\mathcal{D}^{\prime}=\mathrm{C}^{\infty}$
\cite{Ober1}. (More generally, we could have introduced the notion of
$\mathcal{R}$-regularity \cite{ADRapDec}.) We also show, in the spirit of
\cite{JAM4LocAnal}, that some subspaces of $\mathcal{G}$ of regular temperate
elements can be considered leading to the corresponding local analysis of
elements of $\mathcal{G}$.

\section{Simplified or special algebras of generalized functions
\label{DSHCsec2}}

\subsection{Colombeau type algebras based on locally convex
algebras\label{DSHCSubSec2.1}}

Let $d$ be an integer and denote by $\mathbb{K}$ the field of real or complex
numbers. Let $E\left(  \cdot\right)  $ be a presheaf (resp. sheaf) of
$\mathbb{K}$-topological algebras of $\mathbb{K}$ valued functions over
$\mathbb{R}^{d}$.\ (Thus, the presheaf restriction operator is the usual
restriction of $\mathbb{K}$ valued functions.)

Suppose that, for any open set $\Omega$ in $\mathbb{R}^{d}$, the topology of
$E(\Omega)$ can be described by a family $\mathcal{P}(\Omega)=(p_{i})_{i\in
I(\Omega)}$ of semi-norms verifying:
\[
\forall i\in I(\Omega),\ \exists(j,k,C)\in I(\Omega)\times I(\Omega
)\times\mathbb{R}_{+}^{\ast}:\forall f,g\in E(\Omega),\ \ p_{i}(fg)\leq
Cp_{j}(f)p_{k}(g).
\]
Set
\begin{align*}
\mathcal{M}_{(E,\mathcal{P})}(\Omega)  &  =\left\{  \left(  f_{\varepsilon
}\right)  _{\varepsilon}\in E(\Omega)^{\left(  0,1\right]  }\,\left\vert
\,\forall i\in I,\;\exists m\in\mathbb{N}:\ p_{i}\left(  f_{\varepsilon
}\right)  =\mathrm{o}\left(  \varepsilon^{-m}\right)  \;\mathrm{as}%
\;\varepsilon\rightarrow0\right.  \right\}  ,\\
\mathcal{N}_{(E,\mathcal{P})}(\Omega)  &  =\left\{  \left(  f_{\varepsilon
}\right)  _{\varepsilon}\in E(\Omega)^{\left(  0,1\right]  }\,\left\vert
\,\forall i\in I,\;\forall m\in\mathbb{N}:\ p_{i}\left(  f_{\varepsilon
}\right)  =\mathrm{o}\left(  \varepsilon^{m}\right)  \;\mathrm{as}%
\;\varepsilon\rightarrow0\right.  \right\}  .
\end{align*}
(The letter $\mathcal{M}$ (resp. $\mathcal{N}$) stands for moderate (resp.
negligible). In the sequel, we shall omit the precision "$\mathrm{as}%
\;\varepsilon\rightarrow0$.")\smallskip

From \cite{JAM1}, it follows that:

\begin{proposition}
\label{GOMNELLEP1}$~$\newline$\left(  a\right)  $~Suppose that the following
assertion holds:\newline\noindent$\left(  1\right)  $~ For any $\Omega_{1}$
and $\Omega_{2}$, open subsets of $\mathbb{R}^{d}$ with $\Omega_{1}\subset$
$\Omega_{2}$, we have $I(\Omega_{1})\subset I(\Omega_{2})$.\ Moreover, if
$\rho_{1}^{2}$ is the restriction operator $E(\Omega_{2})\rightarrow
E(\Omega_{1})$, then, for each $p_{i}\in\mathcal{P}(\Omega_{1})$, the
semi-norm $\widetilde{p}_{i}=p_{i}\circ\rho_{1}^{2}$ extends $p_{i}$ to
$\mathcal{P}(\Omega_{2}).$\newline Then $\mathcal{M}_{(E,\mathcal{P})}(\cdot)$
is a presheaf of $\mathbb{K}$-algebras and $\mathcal{N}_{(E,\mathcal{P}%
)}(\cdot)$ \textit{a presheaf of ideals of }$\mathcal{M}_{(E,\mathcal{P}%
)}(\cdot)$.\smallskip\newline$\left(  b\right)  $~Suppose that $E(\cdot)$ is a
sheaf of $\mathbb{K}$-topological algebras and that assumption (1) and the
following hold:\newline\noindent$\left(  2\right)  $~For any family
$(\Omega_{h})_{h\in H}$ of open sets in $\mathbb{R}^{d}$ with $\Omega
=\underset{h\in H}{\cup}\Omega_{h}$ and for any $p_{i}\in\mathcal{P}(\Omega)$,
there exist a finite subfamily $\left(  \Omega_{j}\right)  _{1\leq j\leq
n(i)}$ and corresponding semi-norms $p_{j}\in\mathcal{P}(\Omega_{j})$ such
that, for any $u\in E(\Omega)$,%
\[
p_{i}\left(  u\right)  \leq C\max_{1\leq j\leq n(i)}p_{j}(u\left\vert
_{\Omega_{j}}\right.  ),\ \ \ C>0.
\]
Then $\mathcal{M}_{(E,\mathcal{P})}(\cdot)$ is a sheaf of $\mathbb{K}%
$-algebras and $\mathcal{N}_{(E,\mathcal{P})}(\cdot)$ \textit{a sheaf of
ideals of }$\mathcal{M}_{(E,\mathcal{P})}(\cdot)$.
\end{proposition}

\begin{definition}
\label{DSHCDefCTA}For any $\Omega$ open subset of $\mathbb{R}^{d}$, the\emph{
Colombeau type algebra associated} to $E\left(  \Omega\right)  $ is the factor
algebra%
\[
\mathcal{G}\left(  \Omega\right)  =\mathcal{M}_{(E,\mathcal{P})}%
(\cdot)/\mathcal{N}_{(E,\mathcal{P})}(\cdot).
\]

\end{definition}

\begin{proposition}
\label{GOMNELLEP2}\cite{Col1,JAM1}$~$\newline$\left(  a\right)  $~Under
assumption $(1)$, $\mathcal{G}\left(  \cdot\right)  $ is a presheaf of
algebras.\smallskip\newline$\left(  b\right)  $~In addition, suppose that
assumption $(2)$ is fulfilled. Then, the localization principle $\left(
F_{1}\right)  $ holds for $\mathcal{G}\left(  \cdot\right)  $:\newline$\left(
F_{1}\right)  $~Let $(\Omega_{h})_{h\in H}$ be a family of open sets in
$\mathbb{R}^{d}$ with $\Omega=\cup_{h\in H}\Omega_{h}$. Consider
$u,v\in\mathcal{G}(\Omega)$ such that all restrictions $u\left\vert
_{\Omega_{h}}\right.  $ and $v\left\vert _{\Omega_{h}}\right.  $ $\left(  h\in
H\right)  $ coincide.\ Then $u=v$.\smallskip\newline$\left(  c\right)
$~Moreover, if $E\left(  \cdot\right)  $ is a fine sheaf of algebras,
$\mathcal{G}\left(  \cdot\right)  $ is also a fine sheaf of algebras.
\end{proposition}

There is a natural presheaf (resp. sheaf) embedding of $E\left(  \cdot\right)
$ into $\mathcal{G}\left(  \cdot\right)  $ defined by%
\begin{equation}
\sigma_{E,\mathcal{G}}\left(  \Omega\right)  :E\left(  \Omega\right)
\rightarrow\mathcal{G}\left(  \Omega\right)  ,\ \ \ f\mapsto\left(  f\right)
_{\varepsilon}+\mathcal{N}_{(E,\mathcal{P})}(\Omega). \label{GOMNELLEF1}%
\end{equation}
The presheaf (resp. sheaf) $\mathcal{G}\left(  \cdot\right)  $ turns to be a
presheaf (resp. sheaf) of modules on the factor ring $\widetilde{\mathbb{C}%
}=\mathcal{M}\left(  \mathbb{C}\right)  /\mathcal{N}\left(  \mathbb{C}\right)
$ with
\begin{align*}
\mathcal{M}\left(  \mathbb{K}\right)   &  =\left\{  \left(  r_{\varepsilon
}\right)  _{\varepsilon}\in\mathbb{K}^{\left(  0,1\right]  }\,\left\vert
\,\exists m\in\mathbb{N}:\ \left\vert r_{\varepsilon}\right\vert
=\mathrm{o}\left(  \varepsilon^{-m}\right)  \right.  \right\}  ,\\
\mathcal{N}\left(  \mathbb{K}\right)   &  =\left\{  \left(  r_{\varepsilon
}\right)  _{\varepsilon}\in\mathbb{K}^{\left(  0,1\right]  }\,\left\vert
\,\forall m\in\mathbb{N}:\ \left\vert r_{\varepsilon}\right\vert
=\mathrm{o}\left(  \varepsilon^{m}\right)  \right.  \right\}  ,
\end{align*}
with $\mathbb{K}=\mathbb{C}$ or $\mathbb{K}=\mathbb{R},\;\mathbb{R}_{+}$.
Moreover, for the cases under consideration in this paper, $E\left(
\cdot\right)  $ is a subpresheaf (resp. subsheaf) of the sheaf $\mathrm{C}%
^{\infty}\left(  \cdot\right)  $ of smooth functions.\ Then, one easily checks
that, for $\alpha\in\mathbb{N}^{d}$, a presheaf family of differential
operators $\partial^{\alpha}f$ is defined component-wise on $\mathcal{G}%
\left(  \Omega\right)  $ by%
\[
\partial^{\alpha}f=\left(  \partial^{\alpha}f_{\varepsilon}\right)
_{\varepsilon}+\mathcal{N}_{(E,\mathcal{P})}(\Omega)\ \ \mathrm{with}\ \left(
f_{\varepsilon}\right)  _{\varepsilon}\in f.
\]
The family of differential operators $\left(  \partial^{\alpha}\right)
_{\alpha\in\mathbb{N}^{d}}$ satisfies the usual rules (such as the Leibniz
rule) and $\mathcal{G}\left(  \cdot\right)  $ turns to be a presheaf (resp.
sheaf) of differential algebras. The embedding defined by (\ref{GOMNELLEF1})
turns to be an embedding of differential algebras.

\subsection{Examples}

\begin{example}
\label{DSHCexCinf}Take $E\left(  \cdot\right)  =\mathrm{C}^{\infty}\left(
\cdot\right)  $.\ For any $\Omega$ open subset of $\mathbb{R}^{d}$,
$\mathrm{C}^{\infty}\left(  \Omega\right)  $ is endowed with the family of
semi-norms $\mathcal{P}\left(  \Omega\right)  =\left(  p_{K,l}\right)
_{K\Subset\Omega,l\in\mathbb{N}}$ defined by
\[
p_{K,l}(f)=\sup_{\left\vert \alpha\right\vert \leq l,x\in K}\left\vert
\partial^{\alpha}f\left(  x\right)  \right\vert ,
\]
where the notation $K\Subset\Omega$ means that\ the set $K$ is a compact set
included in $\Omega$.\newline We set
\[
\mathcal{M}_{\mathrm{C}^{\infty}}(\cdot)=\mathcal{M}_{(\mathrm{C}^{\infty
},\mathcal{P})}(\cdot),\;\;\mathcal{N}_{\mathrm{C}^{\infty}}(\cdot
)=\mathcal{N}_{(\mathrm{C}^{\infty},\mathcal{P})}(\cdot).
\]
The sheaf $\mathcal{G}\left(  \cdot\right)  =\mathcal{M}_{\mathrm{C}^{\infty}%
}(\cdot)/\mathcal{N}_{\mathrm{C}^{\infty}}(\cdot)$ is the\textbf{\ sheaf of
special or simplified Colombeau algebras of generalized functions}
\cite{Col1,GKOS,Ober1,Scarpa2}.
\end{example}

\begin{example}
\label{DSHCexSob}Take for $E\left(  \cdot\right)  $ the presheaf $H^{\infty
}\left(  \cdot\right)  =\mathcal{D}_{\mathrm{L}_{2}}\left(  \cdot\right)  $,
with
\[
H^{\infty}\left(  \Omega\right)  =\cap_{m\in\mathbb{N}}H^{m}\left(
\Omega\right)  ,\;\;H^{m}\left(  \Omega\right)  =W^{m,2}\left(  \Omega\right)
.
\]
From Sobolev inequalities, it follows that $H^{\infty}\left(  \Omega\right)  $
is continuously embedded into \textrm{C}$^{\infty}\left(  \Omega\right)
$.\ We may suppose a priori that elements of $H^{\infty}\left(  \Omega\right)
$ are \textrm{C}$^{\infty}$. $H^{\infty}\left(  \Omega\right)  $ is endowed
with the family of norms $\mathcal{P}_{\mathrm{L}_{2}}\left(  \Omega\right)
=\left(  \left\Vert \cdot\right\Vert _{m,\Omega}\right)  _{m\in\mathbb{N}}$
$\ $defined by
\[
\left\Vert f\right\Vert _{m,\Omega}=\sup_{\left\vert \alpha\right\vert \leq
m}\left\Vert \partial^{\alpha}f\right\Vert _{L^{2}\left(  \Omega\right)  }.
\]
We set%
\[
\mathcal{M}_{H}(\cdot)=\mathcal{M}_{(H^{\infty},\mathcal{P}_{\mathrm{L}_{2}}%
)}(\cdot),\;\;\mathcal{N}_{H}(\cdot)=\mathcal{N}_{(H^{\infty},\mathcal{P}%
_{\mathrm{L}_{2}})}(\cdot).
\]
The \textbf{presheaf }$\mathcal{G}_{H}\left(  \cdot\right)  =\mathcal{M}%
_{H}(\cdot)/\mathcal{N}_{H}(\cdot)$ is a \textbf{presheaf of Sobolev
generalized functions}.
\end{example}

For the following example, we set for $f\in\mathrm{C}^{\infty}\left(
\Omega\right)  $, $r\in\mathbb{Z}$ and $l\in\mathbb{N}$,%
\[
p_{r,l}(f)=\sup_{x\in\Omega,\ \left\vert \alpha\right\vert \leq l}\left\langle
x\right\rangle ^{r}\left\vert \partial^{\alpha}f\left(  x\right)  \right\vert
\ \ \mathrm{with\ }\left\langle x\right\rangle =(1+\left\vert x\right\vert
^{2})^{1/2}.
\]

\begin{example}
\label{DSHCRapDec}Take $E\left(  \cdot\right)  =\mathcal{S}\left(
\cdot\right)  $, the presheaf of rapidly decreasing smooth functions.\ For any
$\Omega$ open subset of $\mathbb{R}^{d}$, the topology of $\mathcal{S}\left(
\Omega\right)  $ is described by the family of semi-norms $\mathcal{P}%
_{\mathcal{S}}\left(  \Omega\right)  =\left(  p_{q,l}\right)  _{\left(
q,l\right)  \in\mathbb{N}^{2}}$. We set
\[
\mathcal{M}_{\mathcal{S}}(\cdot)=\mathcal{M}_{(\mathcal{S},\mathcal{P}%
_{\mathcal{S}})}(\cdot),\;\;\mathcal{N}_{\mathcal{S}}(\cdot)=\mathcal{N}%
_{(\mathcal{S},\mathcal{P}_{\mathcal{S}})}(\cdot).
\]
The presheaf $\mathcal{G}_{\mathcal{S}}\left(  \cdot\right)  =\mathcal{M}%
_{\mathcal{S}}(\cdot)/\mathcal{N}_{\mathcal{S}}(\cdot)$ is the
\textbf{presheaf of algebras of rapidly decreasing generalized functions}
\cite{ADRapDec,Gar,GaGrOb,Scarpa1,Scarpa3}.
\end{example}

\begin{remark}
More general constructions can be given, for example if $E\left(
\Omega\right)  $ is a projective or inductive limit of topological algebras.
We refer the reader to \cite{DHPV4} for these cases.
\end{remark}

\subsection{Topology on $\mathcal{G}\left(  \cdot\right)  $%
\label{DSHCSubSec2.3}}

We follow \cite{DHPV4} and use the notations of Subsection \ref{DSHCSubSec2.1}%
. Set, for $\left(  f_{\varepsilon}\right)  _{\varepsilon},\,\left(
g_{\varepsilon}\right)  _{\varepsilon}\in E(\Omega)^{\left(  0,1\right]  }$
and $i\in I(\Omega)$,
\[
\left\Vert f_{\varepsilon}\right\Vert _{i}=\limsup_{\varepsilon\rightarrow
0}p_{i}\left(  f_{\varepsilon}\right)  ^{\left\vert \ln\varepsilon\right\vert
^{-1}}%
\]
and $d_{i}(f_{\varepsilon},g_{\varepsilon})=\left\Vert f_{\varepsilon
}-g_{\varepsilon}\right\Vert _{i}$. We get (Proposition-definition 2,
\cite{DHPV4})
\begin{align*}
\mathcal{M}_{(E,\mathcal{P})}(\Omega)  &  =\left\{  \left(  f_{\varepsilon
}\right)  _{\varepsilon}\in E(\Omega)^{\left(  0,1\right]  }\,\left\vert
\,\forall i\in I:\left\Vert f_{\varepsilon}\right\Vert _{i}<+\infty\right.
\right\}  ,\\
\mathcal{N}_{(E,\mathcal{P})}(\Omega)  &  =\left\{  \left(  f_{\varepsilon
}\right)  _{\varepsilon}\in E(\Omega)^{\left(  0,1\right]  }\,\left\vert
\,\forall i\in I:\ \left\Vert f_{\varepsilon}\right\Vert _{i}=0\right.
\right\}  .
\end{align*}
The family $(d_{i})_{i\in I(\Omega)}$ defines a family of ultrapseudometrics
on $\mathcal{M}_{(E,\mathcal{P})}(\Omega)$, inducing on $\mathcal{M}%
_{(E,\mathcal{P})}(\Omega)$ the structure of a topological ring such that the
intersection of neighborhoods of $0$ is equal to $\mathcal{N}_{(E,\mathcal{P}%
)}(\Omega)$. Thus, this topology transfers to the factor space $\mathcal{G}%
\left(  \Omega\right)  $ which turns to be a topological ring, and a
topological algebra other the factor ring $\widetilde{\mathbb{C}}%
=\mathcal{M}\left(  \mathbb{C}\right)  /\mathcal{N}\left(  \mathbb{C}\right)
$. (By setting $\left\Vert r_{\varepsilon}\right\Vert ^{\prime}=\limsup
_{\varepsilon\rightarrow0}\left\vert r_{\varepsilon}\right\vert ^{\left\vert
\ln\varepsilon\right\vert ^{-1}}$, one easily get that
\[
\mathcal{M}\left(  \mathbb{C}\right)  \text{ (resp. }\mathcal{N}\left(
\mathbb{C}\right)  \text{)}=\left\{  \left(  r_{\varepsilon}\right)
_{\varepsilon}\in\mathbb{K}^{\left(  0,1\right]  }\,\left\vert \,\left\Vert
r_{\varepsilon}\right\Vert ^{\prime}<+\infty\ \text{(resp. }\mathcal{N}\left(
\mathbb{C}\right)  =0\text{)}\right.  \right\}  .
\]
This structure turns $\widetilde{\mathbb{C}}$ into a topological ring.)

This topology coincides with the sharp topology, usually defined in terms of
valuations \cite{Scarpa1,Scarpa2}.

\section{Temperate generalized functions\label{DSHCsec3}}

\subsection{Classical construction \cite{GKOS,NePiSc,Scarpa1}}

We recall that
\[
\mathcal{O}_{M}\left(  \Omega\right)  =\left\{  f\in\mathrm{C}^{\infty}\left(
\Omega\right)  \,\left\vert \,\forall l\in\mathbb{N},\;\exists q\in
\mathbb{N}:\ p_{-q,l}(f)<+\infty\right.  \right\}  .
\]
Define
\begin{align*}
\mathcal{M}_{\tau}\left(  \Omega\right)   &  =\left\{  \left(  f_{\varepsilon
}\right)  _{\varepsilon}\in\mathcal{O}_{M}\left(  \Omega\right)  ^{\left(
0,1\right]  }\,\left\vert \,\forall l\in\mathbb{N},\;\exists q\in
\mathbb{N},\;\exists m\in\mathbb{N}:\ p_{-q,l}\left(  f_{\varepsilon}\right)
=\mathrm{o}\left(  \varepsilon^{-m}\right)  \right.  \right\}  ,\\
\mathcal{N}_{\tau}\left(  \Omega\right)   &  =\left\{  \left(  f_{\varepsilon
}\right)  _{\varepsilon}\in\mathcal{O}_{M}\left(  \Omega\right)  ^{\left(
0,1\right]  }\,\left\vert \,\forall l\in\mathbb{N},\;\exists q\in
\mathbb{N},\;\forall m\in\mathbb{N}:\ p_{-q,l}\left(  f_{\varepsilon}\right)
=\mathrm{o}\left(  \varepsilon^{m}\right)  \right.  \right\}  .
\end{align*}
One can show that $\mathcal{M}_{\tau}\left(  \Omega\right)  $ is a subalgebra
of $\mathcal{O}_{M}\left(  \Omega\right)  ^{\left(  0,1\right]  }$ and
$\mathcal{N}_{\tau}\left(  \Omega\right)  $ an ideal of $\mathcal{M}_{\tau
}\left(  \Omega\right)  $. The algebra $\mathcal{G}_{\tau}\left(
\Omega\right)  =\mathcal{M}_{\tau}\left(  \Omega\right)  /\mathcal{N}_{\tau
}\left(  \Omega\right)  $ is called the \emph{algebra of tempered generalized
functions}.

\subsection{New construction}

The topology of $\mathcal{O}_{M}\left(  \Omega\right)  $ may be described by
the non-countable family of semi-norms $\mathcal{P}_{\mathcal{O}_{M}}\left(
\Omega\right)  =\left(  \nu_{\varphi,l}\right)  _{\left(  \varphi,l\right)
\in\mathcal{S}\left(  \Omega\right)  \times\mathbb{N}}$ defined by
\[
\nu_{\varphi,l}(f)=\sup_{x\in\Omega,\left\vert \alpha\right\vert \leq
l}\left\vert \varphi(x)\partial^{\alpha}f(x)\right\vert .
\]

\begin{proposition}
\label{GHTOMAlg}$\mathcal{O}_{M}\left(  \Omega\right)  $ endowed with the
family $\mathcal{P}_{\mathcal{O}_{M}}\left(  \Omega\right)  $ is a topological algebra.
\end{proposition}

This result is classical.\ For the continuity of the product, one establishes
the property%
\[
\forall\left(  \varphi,l\right)  \in\mathcal{S}\left(  \Omega\right)
\times\mathbb{N},\ \exists\psi\in\mathcal{S}\left(  \Omega\right)  ,\ \exists
C>0:\ \forall\left(  f,g\right)  \in\mathcal{O}_{M}\left(  \Omega\right)
^{2},\ \ \nu_{\varphi,l}(fg)\leq C\nu_{\psi,l}(f)\nu_{\psi,l}(g),
\]
which is a consequence of the following:

\begin{lemma}
\label{GHTOMSemi}For any $\psi\in C^{0}\left(  \Omega\right)  $ with positive
values such that, for any $q>0$, $p_{q,0}(\varphi)<+\infty$ there exists
$\varphi\in\mathcal{S}\left(  \Omega\right)  $ such that $\psi\leq\varphi$.
\end{lemma}

With the previous notations, we set%
\begin{align*}
\mathcal{M}_{\mathcal{O}_{M}}(\Omega)  &  =\mathcal{M}_{(\mathcal{O}%
_{M},\mathcal{P}_{\mathcal{O}_{M}})}(\Omega)\\
&  =\left\{  \left(  f_{\varepsilon}\right)  _{\varepsilon}\in\mathcal{O}%
_{M}\left(  \Omega\right)  ^{\left(  0,1\right]  }\,\left\vert \,\forall
\varphi\in\mathcal{S}\left(  \Omega\right)  ,\;\forall l\in\mathbb{N}%
,\;\exists m\in\mathbb{N}:\ \nu_{\varphi,l}(f_{\varepsilon})=\mathrm{o}\left(
\varepsilon^{-m}\right)  \right.  \right\}  ,\\
\mathcal{N}_{\mathcal{O}_{M}}(\Omega)  &  =\mathcal{N}_{(\mathcal{O}%
_{M},\mathcal{P}_{\mathcal{O}_{M}})}(\Omega)\\
&  =\left\{  \left(  f_{\varepsilon}\right)  _{\varepsilon}\in\mathcal{O}%
_{M}\left(  \Omega\right)  ^{\left(  0,1\right]  }\,\left\vert \,\forall
\varphi\in\mathcal{S}\left(  \Omega\right)  ,\;\forall l\in\mathbb{N}%
,\;\forall m\in\mathbb{N}:\ \nu_{\varphi,l}(f_{\varepsilon})=\mathrm{o}\left(
\varepsilon^{m}\right)  \right.  \right\}  .
\end{align*}

\begin{proposition}
\label{GOMNllePPP}We have $\mathcal{M}_{\mathcal{O}_{M}}(\mathbb{R}%
^{d})=\mathcal{M}_{\tau}(\mathbb{R}^{d})$ and $\mathcal{N}_{\mathcal{O}_{M}%
}(\mathbb{R}^{d})=\mathcal{N}_{\tau}(\mathbb{R}^{d}).$
\end{proposition}

\begin{proof}
From the definitions, we immediately get that $\mathcal{M}_{\tau}%
(\mathbb{R}^{d})\subset\mathcal{M}_{\mathcal{O}_{M}}(\mathbb{R}^{d})$ (resp.
$\mathcal{N}_{\tau}(\mathbb{R}^{d})\subset\mathcal{N}_{\mathcal{O}_{M}%
}(\mathbb{R}^{d})$.)\ For the inverse inclusions, we begin by proving that,
for $\left(  f_{\varepsilon}\right)  _{\varepsilon}\in\mathcal{M}%
_{\mathcal{O}_{M}}(\mathbb{R}^{d})$, $\left(  f_{\varepsilon}\right)
_{\varepsilon}\in\mathcal{M}_{\tau}(\mathbb{R}^{d})$ if, and only if, $\left(
f_{\varepsilon}\right)  _{\varepsilon}$ satisfies the following characteristic
property%
\begin{equation}%
\begin{array}
[c]{l}%
\forall\alpha\in\mathbb{N}^{d},\ \exists q\in\mathbb{N},\;\exists
m\in\mathbb{N},\ \exists\varepsilon_{0}\in\left(  0,1\right]  ,\ \exists
r>0:\hspace{1.3in}\\
\multicolumn{1}{r}{\hspace{1.3in}\forall\varepsilon\in\left(  0,\varepsilon
_{0}\right]  ,\ \forall x\notin B\left(  0,r\right)  ,\ \ \left\langle
x\right\rangle ^{-q}\left\vert \partial^{\alpha}f_{\varepsilon}\left(
x\right)  \right\vert \leq\varepsilon^{-m}.}%
\end{array}
\label{GOMNlleFF}%
\end{equation}
Indeed, we can easily see that if $\left(  f_{\varepsilon}\right)
_{\varepsilon}\in\mathcal{M}_{\tau}(\mathbb{R}^{d})$, the property
(\ref{GOMNlleFF}) holds even if $\left(  f_{\varepsilon}\right)
_{\varepsilon}\notin\mathcal{M}_{\mathcal{O}_{M}}(\mathbb{R}^{d})$. Conversely
suppose that $\left(  f_{\varepsilon}\right)  _{\varepsilon}\in\mathcal{M}%
_{\mathcal{O}_{M}}(\mathbb{R}^{d})$ and that (\ref{GOMNlleFF}) holds.\ Fix
$\alpha\in\mathbb{N}^{d}$. There exist $q\in\mathbb{N}$,\ $m\in\mathbb{N}%
$,\ $\varepsilon_{0}\in\left(  0,1\right]  $,\ $r>0$ such that
(\ref{GOMNlleFF}) holds. Let us show that $\left\langle x\right\rangle
^{-q}\left\vert \partial^{\alpha}f_{\varepsilon}\left(  x\right)  \right\vert
\leq\varepsilon^{-m^{\prime}}$ for some $m^{\prime}\in\mathbb{N}$,
$\varepsilon$ small enough and all $x\in B\left(  0,r\right)  $. Consider
$\varphi\in\mathcal{D}\left(  \mathbb{R}^{d}\right)  $ with $0\leq\varphi
\leq1$ and $\varphi\equiv1$ on $B\left(  0,r\right)  $. According to the
definition of $\mathcal{M}_{\mathcal{O}_{M}}(\mathbb{R}^{d})$, used with
$l=\left\vert \alpha\right\vert $, there exists $\varepsilon_{0}^{\prime}%
\in\left(  0,1\right]  $ such that
\[
\forall\varepsilon\in\left(  0,\varepsilon_{0}^{\prime}\right]  ,\ \forall
x\in B\left(  0,r\right)  ,\ \ \left\langle x\right\rangle ^{-q}\left\vert
\partial^{\alpha}f_{\varepsilon}\left(  x\right)  \right\vert \leq\left\vert
\partial^{\alpha}f_{\varepsilon}\left(  x\right)  \right\vert \leq\nu
_{\varphi,l}(f_{\varepsilon})\leq\varepsilon^{-m^{\prime}}.
\]
Taking $\varepsilon_{1}=\min(\varepsilon_{0},\varepsilon_{0}^{\prime})$,
$m_{1}=\max\left(  m,m^{\prime}\right)  $, we obtain that%
\[
\forall\varepsilon\in\left(  0,\varepsilon_{1}\right]  \text{, }\forall
x\in\mathbb{R}^{d},\ \ \left\langle x\right\rangle ^{-q}\left\vert
\partial^{\alpha}f_{\varepsilon}\left(  x\right)  \right\vert \leq
\varepsilon^{-m_{1}}.
\]
From this last property, a classical argument shows that$\ p_{-q,l}\left(
f_{\varepsilon}\right)  =\sup_{x\in\mathbb{R}^{d},\left\vert \alpha\right\vert
\leq l}\left\langle x\right\rangle ^{-q}\left\vert \partial^{\alpha
}f_{\varepsilon}\left(  x\right)  \right\vert =\mathrm{o}\left(
\varepsilon^{-M}\right)  $, provided $M$ is chosen big enough. Thus $\left(
f_{\varepsilon}\right)  _{\varepsilon}\in\mathcal{M}_{\tau}(\mathbb{R}^{d}%
)$.\smallskip

Let us return to the proof of the inclusion $\mathcal{M}_{\mathcal{O}_{M}%
}(\mathbb{R}^{d})\subset\mathcal{M}_{\tau}(\mathbb{R}^{d})$. Take $\left(
f_{\varepsilon}\right)  _{\varepsilon}\in\mathcal{M}_{\mathcal{O}_{M}}%
(\Omega)$ and suppose that (\ref{GOMNlleFF}) does not hold.\ There exist
$\alpha\in\mathbb{N}^{d}$ for which we can built by induction a sequence
$\left(  \varepsilon_{q}\right)  _{q\geq0}$ with $\varepsilon_{q}%
\overset{q\rightarrow+\infty}{\longrightarrow}0$ and a sequence $\left(
x_{q}\right)  _{q\geq0}$ with $\left\vert x_{q+1}\right\vert \geq\left\vert
x_{q}\right\vert +2$ such that
\[
\left\langle x_{q}\right\rangle ^{-q}\left\vert \partial^{\alpha
}f_{\varepsilon_{q}}\left(  x_{q}\right)  \right\vert >\varepsilon_{q}^{-q}.
\]
Consider $\theta\in\mathcal{D}\left(  \mathbb{R}^{d}\right)  $ with
$\operatorname*{supp}\theta\subset B\left(  0,1\right)  $, $0\leq\theta\leq1$
and, say, $\theta\left(  0\right)  =1$. Set
\[
\varphi\left(  x\right)  =\sum_{q=0}^{+\infty}\left\langle x_{q}\right\rangle
^{-q}\theta\left(  x-x_{q}\right)  .
\]
Following \cite{vokhac2}, it can be verified that $\varphi$ belongs to
$\mathcal{S}\left(  \mathbb{R}^{d}\right)  $. (Note that $\operatorname*{supp}%
\left(  x\mapsto\theta\left(  x-x_{q}\right)  \right)  \cap
\operatorname*{supp}\left(  x\mapsto\theta\left(  x-x_{q^{\prime}}\right)
\right)  =\varnothing$ for $q\neq q^{\prime}$, justifying the choice of
$\left(  x_{q}\right)  _{q\geq0}$.) We have%
\[
\varphi\left(  x_{q}\right)  \left\vert \partial^{\alpha}f_{\varepsilon_{q}%
}\left(  x_{q}\right)  \right\vert >\varepsilon_{q}^{-m}\theta\left(
0\right)  =\varepsilon_{q}^{-q}.
\]
Thus, for all $q\in\mathbb{N}$, $\nu_{\varphi,\left\vert \alpha\right\vert
}(f_{\varepsilon_{q}})>\varepsilon_{q}^{-q}$, with $\varepsilon_{q}%
\overset{q\rightarrow+\infty}{\longrightarrow}0$ in contradiction with the
definition of $\mathcal{M}_{\mathcal{O}_{M}}(\Omega)$.\ Finally $\left(
f_{\varepsilon}\right)  _{\varepsilon}\in\mathcal{M}_{\tau}(\mathbb{R}^{d})$.
The proof of the inclusion $\mathcal{N}_{\mathcal{O}_{M}}(\mathbb{R}%
^{d})\subset\mathcal{N}_{\tau}(\mathbb{R}^{d})$ is quite similar.
\end{proof}

\begin{corollary}
\label{GOMNlleCCC}We have $\mathcal{G}_{\tau}(\mathbb{R}^{d})=\mathcal{M}%
_{\tau}(\mathbb{R}^{d})/\mathcal{N}_{\tau}(\mathbb{R}^{d}).$
\end{corollary}

\begin{remark}
\label{GOMnelleGOMLocP-Rem}~\newline$\left(  i\right)  $~Following Subsection
\ref{DSHCSubSec2.3}, $\mathcal{G}_{\tau}(\mathbb{R}^{d})$ is naturally
equipped with a topological structure, given by the non countable family of
ultrapseudometrics $\left(  d_{\varphi,l}\right)  _{\left(  \varphi,l\right)
\in\mathcal{S}\left(  \Omega\right)  \times\mathbb{N}}$ defined by
\[
d_{\varphi,l}(f,g)=\limsup_{\varepsilon\rightarrow0}\nu_{\varphi
,l}(f_{\varepsilon}-g_{\varepsilon})^{\left\vert \ln\varepsilon\right\vert
^{-1}}\text{ where }\left(  f_{\varepsilon}\right)  _{\varepsilon}\in
f,\ \left(  g_{\varepsilon}\right)  _{\varepsilon}\in g.
\]
$\left(  ii\right)  $~According to Proposition \ref{GOMNELLEP1},
$\mathcal{G}_{\tau}(\cdot)$ is a presheaf of algebras. However, the
localization principle $(F_{1})$ does not hold for $\mathcal{G}_{\tau}(\cdot)$
as shown by the following example.
\end{remark}

\begin{example}
\label{GOMnelleGOMLocP-Exa}We adapt a classical example, which was first used
to show that $\mathcal{G}_{\mathcal{\tau}}(\cdot)$ is not a subpresheaf of
$\mathcal{G}(\cdot)$ \cite{Scarpa1}.\ Consider $\Psi\in\mathcal{D}\left(
\mathbb{R}\right)  $ such that $0\leq\Psi\leq1$ and, say, $\Psi\left(
0\right)  =1$. Set $f_{\varepsilon}\left(  \cdot\right)  =\Psi(\cdot
-\left\vert \ln\varepsilon\right\vert ^{1/2})$. Obviously $\left(
f_{\varepsilon}\right)  _{\varepsilon}\in\mathcal{M}_{\mathcal{O}_{M}%
}(\mathbb{R})$, defining $f\in\mathcal{G}_{\tau}(\mathbb{R})$. Consider
$\Omega_{h}=\left]  -h,h\right[  $ for $h\in\mathbb{N}$. As $\left\vert
\ln\varepsilon\right\vert ^{1/2}\overset{\varepsilon\rightarrow0}%
{\longrightarrow}+\infty$,\ we have $f\left\vert _{\Omega_{h}}\right.  =0$.
However, $f_{\varepsilon}(\left\vert \ln\varepsilon\right\vert ^{1/2}%
)=\Psi\left(  0\right)  =1$.\ Take $\varphi\in\mathcal{S}(\mathbb{R})$ defined
by $\varphi\left(  x\right)  =\exp(-x^{2})$. We have $\varphi(\left\vert
\ln\varepsilon\right\vert ^{1/2})f_{\varepsilon}(\left\vert \ln\varepsilon
\right\vert ^{1/2})=\varepsilon$. Thus $\nu_{\varphi,0}(f_{\varepsilon}%
)\geq\varepsilon$ and $\left(  f_{\varepsilon}\right)  _{\varepsilon}%
\notin\mathcal{N}_{\mathcal{O}_{M}}(\mathbb{R})$. Therefore $f$ is non equal
to $0$ on $\mathbb{R=}\cup_{h\in\mathbb{N}}\Omega_{h}$.
\end{example}

We set%
\[
\mathcal{N}_{\mathcal{O}_{M},0}=\left\{  \left(  f_{\varepsilon}\right)
_{\varepsilon}\in\mathcal{M}_{\mathcal{O}_{M}}(\Omega)\,\left\vert
\,\forall\varphi\in\mathcal{S}\left(  \Omega\right)  ,\;\forall m\in
\mathbb{N}:\ \nu_{\varphi,0}\left(  f_{\varepsilon}\right)  =\mathrm{o}\left(
\varepsilon^{m}\right)  \right.  \right\}  .
\]
We have the same result as theorem 1.2.27 in \cite{GKOS}
concerning\ $\mathcal{N}_{\tau}\left(  \cdot\right)  $ (the proof is similar):

\begin{proposition}
\label{GHTTempNstar}If the open set $\Omega$ is a product of $d$ intervals,
$\mathcal{N}_{\mathcal{O}_{M}}\left(  \Omega\right)  $ is equal to
$\mathcal{N}_{\mathcal{O}_{M},0}\left(  \Omega\right)  \cap\mathcal{M}%
_{\mathcal{O}_{M}}\left(  \Omega\right)  $.
\end{proposition}

This result renders easier the proof of the:

\begin{proposition}
\label{GHTTempEmbed}\cite{GKOS,Scarpa2} Consider $\rho\in\mathcal{S}%
(\mathbb{R}^{d})$ such that
\begin{equation}
\int\rho\left(  x\right)  \,\mathrm{d}x=1\ ;\ \ \forall\alpha\in\mathbb{N}%
^{d}\backslash\left\{  0\right\}  ,\ \int x^{\alpha}\rho\left(  x\right)
\,\mathrm{d}x=0. \label{GOMNellePEmb1}%
\end{equation}
Set%
\begin{equation}
\rho_{\varepsilon}\left(  x\right)  =\varepsilon^{-d}\rho\left(
x/\varepsilon^{-1}\right)  . \label{GOMNelleDEmb1}%
\end{equation}
$(i)$~The map
\[
\sigma_{\tau}:\mathcal{O}_{M}(\mathbb{R}^{d})\rightarrow\mathcal{G}_{\tau
}(\mathbb{R}^{d}),\ u\mapsto\left(  u\right)  _{\varepsilon}+\mathcal{N}%
_{\mathcal{O}_{M}}(\mathbb{R}^{d})
\]
is an embedding of differential algebras.\newline$(ii)$~The map
\[
\iota_{\tau}:\mathcal{S}^{\prime}(\mathbb{R}^{d})\rightarrow\mathcal{G}_{\tau
}(\mathbb{R}^{d}),\ T\mapsto\left(  T\ast\rho_{\varepsilon}\right)
_{\varepsilon}+\mathcal{N}_{\mathcal{O}_{M}}(\mathbb{R}^{d})
\]
is an embedding of differential vector spaces.\newline$(iii)$~Moreover,
$\iota_{\tau}|_{\mathcal{O}_{M}(\mathbb{R}^{d})}=\sigma_{\tau}$, which means
that the following diagram is commutative:%
\begin{equation}%
\begin{tabular}
[c]{ccccc}%
$\mathcal{O}_{M}(\mathbb{R}^{d})$ &  & $\overset{\sigma_{\tau}}%
{\longrightarrow}$ &  & $\mathcal{G}_{\tau}(\mathbb{R}^{d}).$\\
& $\searrow$ &  & $^{\iota_{\tau}}\nearrow$ & \\
&  & $\mathcal{S}^{\prime}(\mathbb{R}^{d})$ &  &
\end{tabular}
\ \label{GOMNelleDiag1}%
\end{equation}
(The arrow without name is the usual canonical embedding of $\mathcal{O}%
_{M}(\mathbb{R}^{d})$ into $\mathcal{S}^{\prime}(\mathbb{R}^{d})$.)
\end{proposition}

The assertion $\left(  iii\right)  $ is an improvement of the classical one
which only gives $\iota_{\tau}|_{\mathcal{O}_{M}(\mathbb{R}^{d})}=\sigma
_{\tau}|_{\mathcal{O}_{C}(\mathbb{R}^{d})}$.\medskip

\begin{proof}
The assertion $(i)$ is the application of the general principle recalled in
Subsection \ref{DSHCSubSec2.1} to the case of $\mathcal{O}_{M}\left(
\cdot\right)  $.\ We refer the reader to \cite{GKOS,Scarpa2} for the proof of
the assertion $(ii)$ which uses mainly the structure of elements of
$\mathcal{S}^{\prime}(\mathbb{R}^{d})$. We shall prove the assertion $(iii)$
in the case $d=1$, the general case only differs by more complicate algebraic
expressions. Let $f$ be in $\mathcal{O}_{M}\left(  \mathbb{R}\right)  $ and
set $\Delta=\iota_{\tau}\left(  f\right)  -\sigma_{\tau}\left(  f\right)  .$
One representative of $\Delta$ is given by $\left(  \Delta_{\varepsilon
}:\mathbb{R}\rightarrow\mathcal{M}_{\mathcal{O}_{M}}\left(  \mathbb{R}\right)
\right)  _{\varepsilon}$ with
\begin{multline*}
\Delta_{\varepsilon}\left(  y\right)  =\left(  f\ast\theta_{\varepsilon
}\right)  \left(  y\right)  -f(y)=\int f(y-x)\rho_{\varepsilon}(x)\,\mathrm{d}%
x-f(y)=\\
\int\left(  f(y-x)-f(y)\right)  \rho_{\varepsilon}(x)\,\mathrm{d}%
ux=\int\left(  f(y-\varepsilon u)-f(y)\right)  \rho(u)\,\mathrm{d}u
\end{multline*}
since $\int\rho_{\varepsilon}(x)\,\mathrm{d}x=1$. Let $k$ be a positive
integer. Taylor's formula gives
\[
f(y-\varepsilon u)-f(y)=\sum_{i=1}^{k}\frac{\left(  -\varepsilon u\right)
^{i}}{i!}f^{\left(  i\right)  }\left(  y\right)  +\frac{\left(  -\varepsilon
u\right)  ^{k}}{k!}\int_{0}^{1}f^{\left(  k+1\right)  }\left(  y-\varepsilon
uv\right)  \left(  1-v\right)  ^{k}\,\mathrm{d}v.
\]
Using $\int x^{i}\rho_{\varepsilon}(x)\,\mathrm{d}x=0$, for $i\in\left\{
1,\ldots,k\right\}  $, we get
\[
\Delta_{\varepsilon}(y)=\int\frac{\left(  -\varepsilon u\right)  ^{k}}{k!}%
\int_{0}^{1}f^{\left(  k+1\right)  }\left(  y-\varepsilon uv\right)  \left(
1-v\right)  ^{k}\,\mathrm{d}v\,\rho(u)\,\mathrm{d}u.
\]
As $f\in\mathcal{O}_{M}\left(  \mathbb{R}\right)  $, there exists
$p\in\mathbb{N}$ and $C_{3}>0$ such that $\left\vert f^{\left(  k+1\right)
}\left(  \xi\right)  \right\vert \leq C_{3}\,\left(  1+\left\vert
\xi\right\vert \right)  ^{p}$. Thus%
\[
\forall\left(  u,y\right)  \in\mathbb{R}^{2},\ \forall v\in\left[  0,1\right]
,\ \forall\varepsilon\in\left(  0,1\right]  ,\ \ \left\vert f^{\left(
k+1\right)  }\left(  y-\varepsilon uv\right)  \right\vert \leq C_{3}\,\left(
1+\left\vert y\right\vert \right)  ^{p}\left(  1+\left\vert u\right\vert
\right)  ^{p}.
\]
As $\rho$ is rapidly decreasing, the integral $\int\left\vert u\right\vert
^{k}\left(  1+\left\vert u\right\vert \right)  ^{p}\rho(u)\,\mathrm{d}u$
converges and
\[
\left\vert \Delta_{\varepsilon}(y)\right\vert \leq\frac{\varepsilon^{k}}%
{k!}C_{3}\left(  1+\left\vert y\right\vert \right)  ^{p}\int\left\vert
u\right\vert ^{k}\left(  1+\left\vert u\right\vert \right)  ^{p}%
\rho(u)\,\mathrm{d}u\leq\varepsilon^{k}C_{4}\left(  1+\left\vert y\right\vert
\right)  ^{p}.
\]
Consider $\varphi\in\mathcal{S}\left(  \mathbb{R}\right)  $. The function
$\left(  1+\left\vert \cdot\right\vert \right)  ^{p}\left\vert \varphi\left(
\cdot\right)  \right\vert $ is bounded.\ Thus
\[
\sup_{y\in\mathbb{R}}\left\vert \varphi\left(  y\right)  \Delta_{\varepsilon
}(y)\right\vert =\mathrm{o}\left(  \varepsilon^{k}\right)  \;\text{as
}\varepsilon\rightarrow0.
\]
As $\left(  \Delta_{\varepsilon}\right)  _{\varepsilon}\in\mathcal{M}%
_{\mathcal{O}_{M}}\left(  \mathbb{R}\right)  $ and $\sup_{y\in\mathbb{R}%
}\left\vert \varphi\left(  y\right)  \Delta_{\varepsilon}(y)\right\vert
=\mathrm{o}\left(  \varepsilon^{k}\right)  $, we can conclude without
estimating the derivatives that $\left(  \Delta_{\varepsilon}\right)
_{\varepsilon}\in\mathcal{N}_{\mathcal{O}_{M}}(\mathbb{R})$ by using
Proposition \ref{GHTTempNstar}.
\end{proof}

\section{Fourier Transform and space of rapidly decreasing generalized
distributions\label{DSHCsec4}}

There is no need to recall the importance of spectral analysis, based on the
Fourier transform in the theories of distributions \cite{HorPDOT1} and
Colombeau generalized functions (See, for example,
\cite{ADRapDec,HorDeH,HorKun,Scarpa3}).\ In this section, we first define in a
new way the Fourier transform of elements of $\mathcal{G}_{\tau}%
(\mathbb{R}^{d})$ in relationship with a (new) space of generalized
distributions. \medskip

Classically, the Fourier transform of elements of $\mathcal{G}_{\mathcal{\tau
}}(\Omega)$ is defined with the help of ad hoc cutoff functions
\cite{Scarpa1,Scarpa2}.\ More precisely, one sets%
\[
\forall u\in\mathcal{G}_{\tau}(\Omega),\ \mathcal{F}\left(  u\right)
=\int\operatorname{e}^{-\imath xy}u_{\varepsilon}\left(  y\right)
\widehat{\rho}\left(  \varepsilon y\right)  \,\mathrm{d}y+\mathcal{N}%
_{\mathcal{\tau}}(\mathbb{R}^{d})\ \ \mathrm{with}\ \ \left(  u_{\varepsilon
}\right)  _{\varepsilon}\in u,
\]
where $\rho\in\mathcal{S}(\mathbb{R}^{d})$ satisfies (\ref{GOMNellePEmb1}) so
that $\widehat{\rho}\left(  \varepsilon y\right)  \overset{\varepsilon
\rightarrow0}{\longrightarrow}1$. One shows that this definition makes sense
for $\mathcal{F}\left(  u\right)  $ does not depend on the chosen
representative $\left(  u_{\varepsilon}\right)  _{\varepsilon}\in
u$.\ Analogously, one defines $\mathcal{F}^{-1}$. However, this Fourier
Transform lacks some expected properties. (The reader will find a complete
discussion on this subject in \cite{Scarpa2}.)\medskip

Recalling that $\mathcal{O}_{M}(\mathbb{R}^{d})$ is the Fourier image of
$\mathcal{O}_{C}^{\prime}(\mathbb{R}^{d})$ (and reciprocally), we prefer here
to construct the Fourier transform starting from this fact since
$\mathcal{G}_{\tau}(\mathbb{R}^{d})$ is directly built on $\mathcal{O}%
_{M}(\mathbb{R}^{d})$. In other words, we consider $\mathcal{G}_{\tau
}(\mathbb{R}^{d})$ as a space of multiplicators and we introduce a space of
convolutors, both of them being linked as usual by the Fourier Transform and
its inverse.\medskip

Set%
\begin{align*}
\mathcal{M}_{\mathcal{O}_{C}^{\prime}}(\mathbb{R}^{d})  &  =\left\{  \left(
T_{\varepsilon}\right)  _{\varepsilon}\in\mathcal{O}_{C}^{\prime}%
(\mathbb{R}^{d})^{\left(  0,1\right]  }\,\left\vert \,(\mathcal{F}^{-1}\left(
T_{\varepsilon}\right)  )_{\varepsilon}\in\mathcal{M}_{\mathcal{O}_{M}%
}(\mathbb{R}^{d})\right.  \right\}  ,\\
\mathcal{N}_{\mathcal{O}_{C}^{\prime}}(\mathbb{R}^{d})  &  =\left\{  \left(
T_{\varepsilon}\right)  _{\varepsilon}\in\mathcal{O}_{C}^{\prime}%
(\mathbb{R}^{d})^{\left(  0,1\right]  }\,\left\vert \,(\mathcal{F}^{-1}\left(
T_{\varepsilon}\right)  )_{\varepsilon}\in\mathcal{N}_{\mathcal{O}_{M}%
}(\mathbb{R}^{d})\right.  \right\}  .
\end{align*}
From the linearity of $\mathcal{F}^{-1}$ and the linear properties of the
spaces $\mathcal{M}_{\mathcal{O}_{M}}(\mathbb{R}^{d})$ and $\mathcal{N}%
_{\mathcal{O}_{M}}(\mathbb{R}^{d})$, we immediately get that $\mathcal{M}%
_{\mathcal{O}_{C}^{\prime}}(\mathbb{R}^{d})$ is a $\widetilde{\mathbb{C}}%
$-submodule (resp. $\mathbb{C}$-subvector space) of $\mathcal{O}_{C}^{\prime
}(\mathbb{R}^{d})^{\left(  0,1\right]  }$ and $\mathcal{N}_{\mathcal{O}%
_{C}^{\prime}}(\mathbb{R}^{d})$ a $\widetilde{\mathbb{C}}$-submodule (resp.
$\mathbb{C}$-subvector space) of $\mathcal{M}_{\mathcal{O}_{C}^{\prime}%
}(\mathbb{R}^{d})$.\ 

\begin{definition}
\label{GOMnelledefGOCprime}The factor space $\mathcal{G}_{\mathcal{O}%
_{C}^{\prime}}(\mathbb{R}^{d})=\mathcal{M}_{\mathcal{O}_{C}^{\prime}%
}(\mathbb{R}^{d})/\mathcal{N}_{\mathcal{O}_{C}^{\prime}}(\mathbb{R}^{d})$ is
called the space of \emph{rapidly decreasing generalized distributions}.
\end{definition}

With this previous material, the Fourier transform of elements of
$\mathcal{G}_{\tau}(\mathbb{R}^{d})$ is well defined by%
\[
\forall u\in\mathcal{G}_{\tau}(\mathbb{R}^{d}),\ \mathcal{F}\left(  u\right)
=\mathcal{F}\left(  u_{\varepsilon}\right)  +\mathcal{N}_{\mathcal{O}%
_{C}^{\prime}}(\mathbb{R}^{d})\ \ \mathrm{with}\ \left(  u_{\varepsilon
}\right)  _{\varepsilon}\in u.
\]
The inverse Fourier transform from $\mathcal{G}_{\mathcal{O}_{C}^{\prime}%
}(\mathbb{R}^{d})$ into $\mathcal{G}_{\tau}(\mathbb{R}^{d})$ is defined
analogously. This Fourier transform has the expected properties as they only
have to be verified component-wise.

\begin{proposition}
\label{GOMNellePEmb2}~\newline(i)~The map
\[
\sigma_{\mathcal{O}_{C}^{\prime}}:\mathcal{O}_{C}^{\prime}(\mathbb{R}%
^{d})\rightarrow\mathcal{G}_{\mathcal{O}_{C}^{\prime}}(\mathbb{R}%
^{d}),\ u\mapsto\left(  u\right)  _{\varepsilon}+\mathcal{N}_{\mathcal{O}%
_{C}^{\prime}}(\mathbb{R}^{d})
\]
is an embedding of $\mathbb{C}$-vector spaces.\newline(ii)~Take, as in
Proposition \ref{GHTTempEmbed}, $\rho\in\mathcal{S}(\mathbb{R}^{d})$
satisfying (\ref{GOMNellePEmb1}) and $\left(  \rho_{\varepsilon}\right)
_{\varepsilon}$ defined by (\ref{GOMNelleDEmb1}). The map
\[
\iota_{\mathcal{O}_{C}^{\prime}}:\mathcal{S}_{C}^{\prime}(\mathbb{R}%
^{d})\rightarrow\mathcal{G}_{\mathcal{O}_{C}^{\prime}}(\mathbb{R}%
^{d}),\ T\mapsto\left(  T\widehat{\rho}\left(  \varepsilon\cdot\right)
\right)  _{\varepsilon}+\mathcal{N}_{\mathcal{O}_{C}^{\prime}}(\mathbb{R}%
^{d})
\]
is an embedding of $\mathbb{C}$-vector spaces.
\end{proposition}

The \textbf{proof} of \emph{(i)} is immediate, whereas \emph{(ii)} is obtained
by "taking the Fourier transform image of the diagram (\ref{GOMNelleDiag1})"
in Proposition \ref{GHTTempEmbed}. In fact, the following diagram is
commutative%
\[%
\begin{tabular}
[c]{ccccc}%
$\mathcal{O}_{M}(\mathbb{R}^{d})$ &  & $\overset{\sigma_{\tau}}%
{\longrightarrow}$ &  & $\mathcal{G}_{\tau}(\mathbb{R}^{d})$\\
& $\searrow$ &  & $^{\iota_{\tau}}\nearrow$ & \\
$^{\mathcal{F}}\downarrow\mathcal{\,}\uparrow^{\mathcal{F}^{-1}}$ &  &
$\mathcal{S}^{\prime}(\mathbb{R}^{d})$ &  & $^{\mathcal{F}}\downarrow
\mathcal{\,}\uparrow^{\mathcal{F}^{-1}}$\\
& $^{\mathcal{F}^{-1}}\nearrow$ &  & $^{\iota_{\mathcal{O}_{C}^{\prime}}%
\circ\mathcal{F}}\searrow$ & \\
$\mathcal{O}_{C}^{\prime}(\mathbb{R}^{d})$ &  & $\overset{\sigma
_{\mathcal{O}_{C}^{\prime}}}{\longrightarrow}$ &  & $\mathcal{G}%
_{\mathcal{O}_{C}^{\prime}}(\mathbb{R}^{d})$%
\end{tabular}
\ \
\]
(The arrow without name is the usual canonical embedding of $\mathcal{O}%
_{M}(\mathbb{R}^{d})$ into $\mathcal{S}^{\prime}(\mathbb{R}^{d})$.)

\begin{remark}
\label{GOMnelleRemHStype}Following ideas of Jean-Andr\'{e} Marti (private
communication), the Fourier transform in $\mathcal{G}_{\tau}(\mathbb{R}^{d})$
can be used to define Sobolev type subspaces of $\mathcal{G}_{\tau}%
(\mathbb{R}^{d})$.\ More precisely, we say that $\left(  u_{\varepsilon
}\right)  _{\varepsilon}\in\mathcal{M}_{\mathcal{O}_{M}}(\mathbb{R}^{d})$ is
of $H^{s}$\emph{ type} if, for all $\varepsilon\in\left(  0,1\right]  $,
$\left\langle \cdot\right\rangle ^{s}(\cdot)\widehat{u}_{\varepsilon}%
(\cdot)\in\mathrm{L}^{2}(\mathbb{R}^{d})$ and $\left(  \left\Vert \left\langle
\cdot\right\rangle ^{s}\widehat{u}_{\varepsilon}(\cdot)\right\Vert
_{\mathrm{L}^{2}}\right)  _{\varepsilon}\in\mathcal{M}\left(  \mathbb{R}%
\right)  $. One shows that the space $H^{s}(\mathbb{R}^{d})$ is embedded into
$\mathcal{G}_{\tau}^{(s)}(\mathbb{R}^{d})$ through $\iota_{\tau}$ defined in
Proposition \ref{GHTTempEmbed}. This will be used in a forthcoming paper to
introduce a $H^{s}$ local and microlocal analysis in spaces of generalized functions.
\end{remark}

\section{Introduction to regularity theory\label{DSHCsec5}}

\subsection{The spaces $\mathcal{M}_{\tau}^{\infty}(\Omega)$ and
$\mathcal{G}_{\tau}^{\infty}(\Omega)$}

In analogy to the definition of $\mathcal{G}^{\infty}$ \cite{GKOS,Ober1}, we
set%
\[
\mathcal{M}_{\mathcal{O}_{M}}^{\infty}(\Omega)=\left\{  \left(  f_{\varepsilon
}\right)  _{\varepsilon}\in\mathcal{O}_{M}\left(  \Omega\right)  ^{\left(
0,1\right]  }\,\left\vert \,\forall\varphi\in\mathcal{S}\left(  \Omega\right)
,\ \exists m\in\mathbb{N},\ \forall l\in\mathbb{N}:\nu_{\varphi,l}%
(f_{\varepsilon})=\mathrm{o}\left(  \varepsilon^{-m}\right)  \right.
\right\}  .
\]
It is easy to check that $\mathcal{M}_{\mathcal{O}_{M}}^{\infty}\left(
\cdot\right)  $ is a subpreasheaf of algebras of $\mathcal{M}_{\mathcal{O}%
_{M}}\left(  \cdot\right)  $.\ From this, we get the:

\begin{proposition}
\label{GOMNELLEPinfty1}$\mathcal{G}_{\tau}^{\infty}\left(  \cdot\right)
=\mathcal{M}_{\mathcal{O}_{M}}^{\infty}\left(  \cdot\right)  /\mathcal{N}%
_{\mathcal{O}_{M}}^{\infty}\left(  \cdot\right)  $ is a subpresheaf of
differential algebras of $\mathcal{G}_{\tau}\left(  \cdot\right)  $.
\end{proposition}

Going further with the above mentioned analogy, we recall that $\mathcal{G}%
^{\infty}(\mathbb{R}^{d})\cap\mathcal{D}^{\prime}(\mathbb{R}^{d}%
)=\mathrm{C}^{\infty}(\mathbb{R}^{d})$ \cite{Ober1}.\ This result can be
interpreted as follows: The subsheaf $\mathcal{G}^{\infty}$ of regular
sections of $\mathcal{G}$ is such that the sheaf embedding $\mathcal{G}%
^{\infty}\rightarrow\mathcal{G}$ is the natural extension of the classical one
$\mathrm{C}^{\infty}\rightarrow\mathcal{D}^{\prime}$. We have here the same
situation (\emph{modulo} the fact $\mathcal{G}_{\tau}\left(  \cdot\right)  $
is only a presheaf) that given by the:

\begin{proposition}
\label{GOMNELLEPinter}$\mathcal{G}_{\tau}^{\infty}(\mathbb{R}^{d}%
)\cap\mathcal{S}^{\prime}(\mathbb{R}^{d})=\mathcal{O}_{M}(\mathbb{R}^{d}).$
\end{proposition}

The result should be understood as follows.\ For $u\in\mathcal{S}^{\prime
}(\mathbb{R}^{d})$, if $\iota_{\tau}\left(  u\right)  $ is in $\mathcal{G}%
_{\tau}^{\infty}(\mathbb{R}^{d})$, then $u$ is in $\mathcal{O}_{M}%
(\mathbb{R}^{d})$.\medskip

\begin{proof}
Take $u\in\mathcal{S}^{\prime}(\mathbb{R}^{d})$ such that $\iota_{\tau}\left(
u\right)  $ is in $\mathcal{G}_{\tau}^{\infty}(\mathbb{R}^{d})$.\ Then
$\left(  u\ast\rho_{\varepsilon}\right)  _{\varepsilon}$ is in $\mathcal{M}%
_{\mathcal{O}_{M}}^{\infty}(\mathbb{R}^{d})$. Recall that%
\[
u\in\mathcal{O}_{M}(\mathbb{R}^{d})\Leftrightarrow\widehat{u}\in
\mathcal{O}_{C}^{\prime}(\mathbb{R}^{d})\Leftrightarrow\forall\psi
\in\mathcal{S}(\mathbb{R}^{d}),\ \widehat{u}\ast\psi\in\mathcal{S}%
(\mathbb{R}^{d}).
\]
Thus consider $\psi\in\mathcal{S}(\mathbb{R}^{d})$. We are going to show that
$\left\langle \cdot\right\rangle ^{m}\widehat{u}\ast\psi\left(  \cdot\right)
$ is bounded for all $m\in\mathbb{N}$. We have
\begin{equation}
\widehat{u}\ast\psi=(\widehat{u}(1-\widehat{\rho_{\varepsilon}}))\ast
\psi+(\widehat{u}\,\widehat{\rho_{\varepsilon}})\ast\psi. \label{GOMnelleWXB0}%
\end{equation}
Recalling that $\rho_{\varepsilon}=\varepsilon^{-d}\rho\left(  \cdot
/\varepsilon\right)  $, we easily get that $\widehat{\rho_{\varepsilon}%
}\left(  \cdot\right)  =\widehat{\rho}\left(  \varepsilon\cdot\right)
$.\ Note also that $\widehat{\rho}\left(  0\right)  =\int\rho\left(  x\right)
\,\mathrm{d}x=1$.\ Thus%
\[
1-\widehat{\rho_{\varepsilon}}\left(  x\right)  =-\varepsilon\int_{0}%
^{1}\nabla\widehat{\rho}\left(  \varepsilon xt\right)  \cdot x\,\mathrm{d}%
t=\varepsilon B\left(  \varepsilon,x\right)  .
\]
As $\widehat{\rho}$ is rapidly decreasing, there exists $C>0$ such that
$\left\vert \nabla\widehat{\rho}\left(  \varepsilon xt\right)  \cdot
x\right\vert \leq C\left\langle x\right\rangle $ for all $\left(
\varepsilon,x,t\right)  \in\left(  0,1\right]  \times\mathbb{R}^{d}%
\times\left[  0,1\right]  $. The same holds for the derivatives with respect
to $x$ and, thus, for the function $B$ and its derivatives. From this, for
example by using the structure of elements of $\mathcal{S}^{\prime}%
(\mathbb{R}^{d})$, it can be shown that $\widehat{u}(1-\widehat{\rho
_{\varepsilon}}))\ast\psi$ satisfies%
\begin{equation}
\forall x\in\mathbb{R}^{d},\ \ \left\vert ((\widehat{u}(1-\widehat
{\rho_{\varepsilon}}))\ast\psi)(x)\right\vert \leq C_{0}\varepsilon
\,\left\langle x\right\rangle ^{q}, \label{GOMnelleWXB1}%
\end{equation}
for some $C_{0}>0$ and $q$ not depending on $\varepsilon$.\smallskip

Consider $l\in\mathbb{N}$ and $\beta\in\mathbb{N}^{d}$ with $\left\vert
\beta\right\vert =l$. We have, for all $x\in\mathbb{R}^{d},$
\[
\left(  \imath x\right)  ^{\beta}((\widehat{u}\,\widehat{\rho_{\varepsilon}%
})\ast\psi)(x)=\left(  \imath x\right)  ^{\beta}\mathcal{F}\left(  \left(
u\ast\rho_{\varepsilon}\right)  \mathcal{F}^{-1}\left(  \psi\right)  \right)
(x)=\mathcal{F}\left(  (\partial^{\beta}\left(  u\ast\rho_{\varepsilon
}\right)  )\,\mathcal{F}^{-1}\left(  \psi\right)  \right)  (x).
\]
Applying the definition of $\mathcal{M}_{\mathcal{O}_{M}}^{\infty}%
(\mathbb{R}^{d})$ for $\left(  \rho_{\varepsilon}\ast u\right)  _{\varepsilon
}$ with $\varphi=\left\langle \cdot\right\rangle ^{\left(  d+1\right)
/2}\mathcal{F}^{-1}\left(  \psi\right)  $, we get the existence of $N$ (only
depending on $\left(  \rho_{\varepsilon}\ast u\right)  _{\varepsilon}$ and
$\psi$) and $C_{1}>0$ such that
\[
\forall y\in\mathbb{R}^{d},\ \ \left\vert (\partial^{\beta}\left(  u\ast
\rho_{\varepsilon}\right)  )(y)\mathcal{F}^{-1}\left(  \psi\right)
(y)\right\vert \leq C_{1}\left\langle y\right\rangle ^{-\left(  d+1\right)
/2}\varepsilon^{-N}\ \ \text{for }\varepsilon\text{ small enough.}%
\]
Thus, we get the existence of $C_{2}>0$ such that
\[
\left\vert \left(  \imath x\right)  ^{\beta}((\widehat{u}\,\widehat
{\rho_{\varepsilon}})\ast\psi)(x)\right\vert =\left\vert \mathcal{F}\left(
(\partial^{\beta}\left(  u\ast\rho_{\varepsilon}\right)  )\,\mathcal{F}%
^{-1}\left(  \psi\right)  \right)  \left(  x\right)  \right\vert \leq
C_{2}\,\varepsilon^{-N}%
\]
for $\varepsilon$ small enough and all $x\in\mathbb{R}^{d}$. Using a classical
argument, we get a constant $C_{3}>0$ such that
\[
\forall x\in\mathbb{R}^{d},\ \ \left\langle x\right\rangle ^{l}\left\vert
(\widehat{u}\,\widehat{\rho_{\varepsilon}})\ast\psi)\left(  x\right)
\right\vert \leq C_{3}\,\varepsilon^{-N}\text{ \ for }\varepsilon\text{ small
enough.}\
\]

\noindent Fix $m\in\mathbb{N}$ and take $l=m+\left(  m+q\right)  N$. Writing
the previous inequality in the form $\left\langle x\right\rangle
^{m}\,\left\vert (\widehat{u}\,\widehat{\rho_{\varepsilon}})\ast\psi)\left(
x\right)  \right\vert \leq C_{3}\,\left(  \varepsilon\left\langle
x\right\rangle ^{m+q}\right)  ^{-N}$, using (\ref{GOMnelleWXB1}) and finally
inserting these intermediates steps in (\ref{GOMnelleWXB0}), we get
\[
\forall x\in\mathbb{R}^{d},\ \ \left\vert \left\langle x\right\rangle
^{m}\left(  \widehat{u}\ast\psi\right)  \left(  x\right)  \right\vert
=C\,(\varepsilon\left\langle x\right\rangle ^{m+q}+\,\left(  \varepsilon
\left\langle x\right\rangle ^{m+q}\right)  ^{-N})=T\left(  \varepsilon
\left\langle x\right\rangle ^{m+q}\right)
\]
for $\varepsilon$ smaller than some $\varepsilon_{0}$ and some $C>0.$Thus, for
$x$ such that $\left\langle x\right\rangle ^{m+q}\geq\varepsilon_{0}^{-1}$,
take $\varepsilon_{x}$ such that $\varepsilon_{x}=\left\langle x\right\rangle
^{-m-q}$ to obtain that $\ \left\vert \left\langle x\right\rangle ^{m}\left(
\widehat{u}\ast\psi\right)  \left(  x\right)  \right\vert \leq T\left(
1\right)  $. From this, it follows that the function $\left\vert \left\langle
\cdot\right\rangle ^{m}\left(  \widehat{u}\ast\psi\right)  \right\vert $ is
bounded on $\mathbb{R}^{d}$, as claimed.
\end{proof}

\subsection{Regularities for temperate generalized functions}

As in the presheaf $\mathcal{G}_{\tau}\left(  \cdot\right)  $ the localization
principle $\left(  F_{1}\right)  $ is not fulfilled, we are not in the
situation to apply the results of \cite{JAM4LocAnal} concerning singular
supports and their properties. Indeed, following the notations of the quoted
paper, we need a presheaf $\mathcal{A}\left(  \cdot\right)  $ (of vector
spaces, of algebras,\ldots) with localization principle and a subpresheaf
$\mathcal{B}\left(  \cdot\right)  $ of $\mathcal{A}\left(  \cdot\right)  $ to
define the $\mathcal{B}$-singular support of a section $u\in\mathcal{A}\left(
\Omega\right)  $. Thus, as it is done in \cite{JAM3} for the definition of the
presheaf $\mathcal{G}^{L}\left(  \cdot\right)  $, we shall start from the
sheaf $\mathcal{G}\left(  \cdot\right)  $ and define some regular
subpresheaves of it. More precisely, for the two cases $\mathcal{B}\left(
\cdot\right)  =\mathcal{G}_{\tau}\left(  \cdot\right)  $, $\mathcal{G}_{\tau
}^{\infty}\left(  \cdot\right)  $, we set%
\[
\mathcal{N}_{\mathcal{O}_{M},\ast}^{\sharp}\left(  \cdot\right)
=\mathcal{N}\left(  \cdot\right)  \cap\mathcal{M}_{\mathcal{O}_{M}}^{\sharp
}\left(  \cdot\right)  ,
\]
where the symbol $"\sharp"$ means successively the blank character and
$"\infty"$. According to the results recalled in Section \ref{DSHCsec2} and to
the inclusion $\mathcal{M}_{\mathcal{O}_{M}}(\mathbb{R}^{d})\subset
\mathcal{M}_{\mathrm{C}^{\infty}}\left(  \mathbb{R}^{d}\right)  $),
$\mathcal{G}_{\tau,\ast}^{\sharp}\left(  \cdot\right)  =\mathcal{M}%
_{\mathcal{O}_{M}}^{\sharp}\left(  \cdot\right)  /\mathcal{N}_{\mathcal{O}%
_{M},\ast}^{\sharp}\left(  \cdot\right)  $ is a subpresheaf of $\mathcal{G}%
\left(  \cdot\right)  $.\ Using the framework and the results of
\cite{JAM4LocAnal}, we say that the elements of $\mathcal{G}_{\tau,\ast
}^{\sharp}\left(  \Omega\right)  $ are $\mathcal{G}_{\tau}^{\sharp}%
$\emph{-regular elements} of $\mathcal{G}\left(  \Omega\right)  $. For
$u\in\mathcal{G}(\Omega)$, we can define $\mathcal{O}_{\tau}^{\sharp}\left(
u\right)  $, the set of all $x\in\Omega$ such that $u$ is $\mathcal{G}_{\tau
}^{\sharp}$\emph{-regular} at $x$, that is%
\[
\mathcal{O}_{\tau}^{\sharp}\left(  u\right)  =\left\{  x\in\Omega,\ \exists
V\in\mathcal{V}_{x}:u\left\vert _{V}\right.  \in\mathcal{G}_{\mathcal{O}%
_{M,\ast}}^{\sharp}\left(  V\right)  \right\}
\]
($\mathcal{V}_{x}$ being the family of all the open neighborhood of $x$.) The
$\mathcal{G}_{\tau}^{\sharp}$-\emph{singular support} of $u$ is the well
defined set $\mathcal{S}_{\tau}^{\sharp}\left(  u\right)
=\operatorname{singsupp}_{\mathcal{G}_{\tau}^{\sharp}}u=\Omega\setminus
\mathcal{O}_{\tau}^{\sharp}\left(  u\right)  $ and has the following
properties \cite{JAM4LocAnal}:

\begin{proposition}
\label{GOMNvllePropSSup1}Consider $u,v\in$ $\mathcal{G}\left(  \Omega\right)
$, $\alpha$ in $\mathbb{N}^{d}$ and $g$ in $\mathcal{G}_{\tau,\ast}^{\sharp
}(\Omega)$.\ We have:\newline(i)~$\mathcal{S}_{\tau}^{\sharp}(u+v)\subset
\mathcal{S}_{\tau}^{\sharp}\left(  u\right)  \cup\mathcal{S}_{\tau}^{\sharp
}(v)\,$;\newline(ii)~$\mathcal{S}_{\tau}^{\sharp}(uv)\subset\mathcal{S}_{\tau
}^{\sharp}(u)\cup\mathcal{S}_{\tau}^{\sharp}(v)\,$;\newline(iii)~$\mathcal{S}%
_{\tau}^{\sharp}\left(  \partial^{\alpha}u\right)  \subset\mathcal{S}_{\tau
}^{\sharp}\left(  u\right)  \,$;\newline(iv)~$\mathcal{S}_{\tau}^{\sharp
}\left(  gu\right)  \subset\mathcal{S}_{\tau}^{\sharp}\left(  u\right)  $.
\end{proposition}

From these properties, one easily gets:

\begin{corollary}
\label{GOMNvllePropSSup4}\textit{Let }$P(\partial)=%
%TCIMACRO{\dsum \limits_{\left\vert \alpha\right\vert \leq m}}%
%BeginExpansion
{\displaystyle\sum\limits_{\left\vert \alpha\right\vert \leq m}}
%EndExpansion
C_{\alpha}\partial^{\alpha}$\textit{\ be a differential polynomial with
coefficients in }$\mathcal{G}_{\tau,\ast}(\Omega)$ (resp. $\mathcal{G}%
_{\tau,\ast}^{\infty}(\Omega)$)$.$ \textit{For any }$u\in\mathcal{G}(\Omega
)$\textit{, we have}%
\[
\mathcal{S}_{\tau}\left(  P(\partial)u\right)  \subset\mathcal{S}_{\tau
}\left(  u\right)  \ (\mathrm{resp.}\ \mathcal{S}_{\tau}^{\infty}\left(
P(\partial)u\right)  \subset\mathcal{S}_{\tau}^{\infty}\left(  u\right)  \,).
\]

\end{corollary}

~

\end{document}